\documentclass[11pt,reqno]{amsart}

\usepackage{enumerate, amsmath, amsthm, amsfonts, amssymb, 
mathrsfs, graphicx, paralist, ulem}
\usepackage[usenames, dvipsnames]{color}
\usepackage[margin=1in]{geometry} 
\usepackage{hyperref}
\usepackage{comment}
\usepackage{bbm}
\usepackage{wrapfig}
\usepackage{mathtools}
\usepackage{mathabx} 
\usepackage{breqn} 
\usepackage{tikz-cd}
\usepackage{amsrefs}

\usetikzlibrary{angles}
\usetikzlibrary{arrows.meta} 
\usetikzlibrary{decorations}
\usetikzlibrary{calc} 

\usepackage{bbm}

\numberwithin{equation}{section}
\newtheorem{Theorem}[equation]{Theorem}
\newtheorem{Proposition}[equation]{Proposition}
\newtheorem{Lemma}[equation]{Lemma}

\newtheorem{Conjecture}[equation]{Conjecture}

\theoremstyle{definition}
\newtheorem{Remark}[equation]{Remark}

\newtheorem{eg}[equation]{Example}
\newenvironment{example}[1][]{\begin{eg}[#1] \pushQED{\qed}}{\popQED \end{eg}}
\newtheorem{Definition}[equation]{Definition}

\newcommand{\cA}{\mathcal{A}}

\newcommand{\bC}{\mathbf{C}}
\newcommand{\cC}{\mathcal{C}}

\newcommand{\cD}{\mathcal{D}}

\newcommand{\cH}{\mathcal{H}}

\newcommand{\cK}{\mathcal{K}}

\newcommand{\cO}{\mathcal{O}}

\newcommand{\cP}{\mathcal{P}}

\newcommand{\cS}{\mathcal{S}}

\newcommand{\cT}{\mathcal{T}}

\newcommand{\cU}{\mathcal{U}}

\newcommand{\cZ}{\mathcal{Z}}

\newcommand{\fg}{\mathfrak{g}}

\newcommand{\fh}{\mathfrak{h}}

\newcommand{\kk}{\mathbbm{k}}

\newcommand{\BB}{\mathsf{B}}
\newcommand{\CC}{\mathbb{C}}

\newcommand{\HH}{\mathbb{H}}

\newcommand{\PP}{\mathbb{P}}
\newcommand{\QQ}{\mathbb{Q}}
\newcommand{\RR}{\mathbb{R}}

\newcommand{\WW}{\mathsf{W}}

\newcommand{\ZZ}{\mathbb{Z}}
\newcommand{\one}{\mathbbm{1}}

\renewcommand{\phi}{\varphi}

\renewcommand{\tilde}[1]{\widetilde{#1}}
\newcommand{\ol}[1]{\overline{#1}}

\newcommand{\todo}[1]{\textcolor{blue}{$[\star$ To do: #1 $\star]$}}

\newcommand{\oded}[1]{\textcolor{red}{$[\star$ Oded: #1 $\star]$}}

\renewcommand{\hom}{\operatorname{Hom}}

\DeclareMathOperator{\rank}{rank}

\DeclareMathOperator{\reg}{reg}

\DeclareMathOperator{\Aut}{Aut}

\DeclareMathOperator{\Stab}{Stab}
\DeclareMathOperator{\Path}{Path}

\DeclareMathOperator{\Conf}{Conf}
\DeclareMathOperator{\regg}{reg}

\DeclareMathOperator{\std}{std}

\newcommand{\mapsfrom}{\mathrel{\reflectbox{\ensuremath{\mapsto}}}}

\DeclareMathOperator{\Deck}{Deck}

\newcommand{\GL}{GL}

\newcommand{\mmod}{\mathrm{-mod}}
\newcommand{\proj}{\mathrm{-proj}}
\newcommand{\dgmod}{\mathrm{-dgmod}}

\newcommand{\catname}[1]{{\mathsf{#1}}}
\newcommand{\Fib}{\catname{Fib}}
\DeclareMathOperator{\FPdim}{FPdim}
\DeclareMathOperator{\Irr}{Irr}
\renewcommand{\vec}{\catname{vec}}

\makeatletter
\newcommand{\xMapsto}[2][]{\ext@arrow 0599{\Mapstofill@}{#1}{#2}}
\def\Mapstofill@{\arrowfill@{\Mapstochar\Relbar}\Relbar\Rightarrow}
\makeatother

\newcommand\iso\cong
\newcommand\into\hookrightarrow
\newcommand\onto\twoheadrightarrow
\newcommand{\Hom}{\operatorname{Hom}}
\newcommand{\nc}{\newcommand}
\nc{\la}{\lambda}
\nc{\Iso}{\mathsf{Iso}}
\nc{\Id}{\mathrm{Id}}

\begin{document}

\title[Periodic elements and stability conditions]{Periodic elements in finite type Artin--Tits groups \\ and stability conditions}

\author{Edmund Heng}
\address{E.~Heng: School of Mathematics and Statistics, University of Sydney, Australia}
\email{edmund.heng@sydney.edu.au}
\author{Anthony Licata}
\address{A.~Licata: CNRS-ANU International Research Laboratory FAMSI and Mathematical Sciences Institute, The Australian National University, Australia}
\email{anthony.licata@anu.edu.au}
\author{Oded Yacobi}
\address{O.~Yacobi: School of Mathematics and Statistics, University of Sydney, Australia}
\email{oded.yacobi@sydney.edu.au}
\date{}

\begin{abstract}
Periodic elements in finite type Artin--Tits groups are elements some positive power of which is central.  We give a dynamical characterisation of periodic elements via their action on the corresponding 2-Calabi--Yau category and on its space of (fusion equivariant) Bridgeland stability conditions.  The main theorem is that an element $\beta$ is periodic if and only if $\beta$ has a fixed point in the stability manifold.
\end{abstract}

\maketitle

\section{Introduction}

Let $(\WW,S)$ be a finite type Coxeter system, with associated Artin--Tits group $\BB$.  The group $\BB$ acts faithfully by autoequivalences on a 2-Calabi-Yau category $\cT$.  This realises the Artin--Tits group $\BB$ as a spherical twist group, a categorical analog of a reflection group.  This realisation is important in part because it allows dynamics to be introduced into the study of Artin--Tits groups, parallel to other settings in geometric group theory.  In particular, the action of $\BB$ on the space of Bridgeland stability conditions plays a central role in this dynamical study (see Section \ref{sec:final} for more details).  The goal of this note is to give a dynamical description of periodic elements of $\BB$ -- those elements $\beta\in \BB$ such that some non-zero power of $\beta$ is central --  using the action of $\BB$ on the space of stability conditions. The main theorem is as follows.

\begin{Theorem}\label{thm:main2}
A element $\beta \in \BB$ is periodic if and only if $\beta$ has a fixed point in $\Stab_\cC(\cT)/\CC$.
\end{Theorem}

In the above theorem, the space $\Stab_\cC(\cT)$ is a distinguished connected component of the space of fusion-equivariant stability conditions on $\cT$, considered in \cite{HL24}, and $\Stab_\cC(\cT)/\CC$ is the quotient by the natural $\CC$ action.  In simply-laced type ADE, the space $\Stab_\cC(\cT)$ is the classical space of stability conditions which appears in the works of Thomas \cite{Thomas06}, Bridgeland \cite{Bridgeland09} and Brav-Thomas \cite{BravThomas}.  Theorem \ref{thm:main2} gives an analog for finite type Artin groups of the characterisation of periodic elements in mapping class groups as those mapping classes which admit a fixed point in Teichm\"uller space. 

The proof of Theorem \ref{thm:main2} relies on two essential inputs.  The first is Bessis's ``Springer theory for braid groups'' \cite{Bes15} which allows us to reduce the study of arbitrary periodic elements of $\BB$ to that of $d$-th roots of the full twist, where the positive integer $d$ is regular in the sense of Springer \cite{Springer74}.  In particular, Springer theory provides us with the central charge of the sought-after fixed point for such a $d$-th root.  The second input is the description of $\Stab_\cC(\cT)$ as the universal cover of the space of central charges, which is identified with the complexified hyperplane complement studied classically in Coxeter theory \cites{Thomas06, Bridgeland09, HL24}.  This covering property is our main motivation for using 2-Calabi--Yau categories (as opposed to other categories with an action of $\BB$) in the dynamical study of $\BB$.  In fact, our argument for the forward implication of the main theorem does not make fundamental use of the modular description of the universal cover as a stability manifold, whereas the backward implication does (see Remarks \ref{rem:direct1} and \ref{rem:direct2}, respectively).

A dynamical classification of arbitrary elements of $\BB$ also requires understanding the pseudo-Anosov elements, which we expect to exhibit sink-source dynamics on the boundary of $\Stab_\cC(\cT)/\CC$, and which should have positive entropy in the sense of \cite{DHKK}.  Categorical frameworks for both pseudo-Anosov autoequivalences and for the boundary of $\Stab_\cC(\cT)/\CC$ have been considered recently (see e.g. \cites{DHKK, FFHKL, BDL} and Section \ref{sec:final}).

\subsection*{Acknowledgement}
O.Y. thanks MPIM-Bonn, and all three authors thank IH\'{E}S, where part of this work was undertaken. O.Y. and A.L. acknowledge support from Discovery Project grants from the Australian Research Council. 

\section{Background}

\subsection{Artin--Tits groups of finite type Coxeter systems} Let $\Gamma$ be finite type Coxeter graph, that is, one whose connected components are all of type $A,B=C,D,E,F,G,H,I$. Denote the set of vertices of $\Gamma$ by $\Gamma_0$.  There is an $m_{s,t}$-labeled edge connecting $s, t\in \Gamma_0$ whenever $m_{s,t} \geq 3$, and there is no edge between vertices $s,t\in \Gamma_0$ when $m_{s,t}=2$.  

The \textbf{Coxeter group} $\WW$ of $\Gamma$ is a finite group with presentation
\[
\WW=\langle s\in \Gamma_0 \mid (\forall s,t \in \Gamma_0)\;  s^2=1,\; (st)^{m_{s,t}}=1 \rangle.
\]
The \textbf{Artin--Tits} group $\BB$ is presented as
\begin{align}\label{eq:braidrels}
\BB=\langle \sigma_s \text{ for each }s \in \Gamma_0 \mid (\forall s,t \in \Gamma_0)\; \sigma_s\sigma_t\sigma_s\cdots = \sigma_t\sigma_s\sigma_t\cdots  \rangle,
\end{align}
with $m_{s,t}$ terms on either side of the equation. 

The surjective group homomorphism $\BB \to \WW$ mapping $\sigma_s$ to $s$ admits a set-theoretic section (positive lift) $w \mapsto \sigma_w$, where if $w=st\cdots $ is a reduced expression then $\sigma_w \coloneqq \sigma_{s}\sigma_t\cdots$. 
In particular, let $w_0 \in \WW$ denote the longest element and set $\Delta \coloneqq \sigma_{w_0}$.

The real vector space $V_\RR=\bigoplus_{s \in \Gamma_0}\RR\alpha_s$ has a symmetric bilinear form $(\cdot,\cdot):V_\RR \times V_\RR \to \RR$ given by 
\[
(\alpha_s,\alpha_s)=2, \ \ (\alpha_s,\alpha_t)=-2\cos\left(\frac{\pi}{m_{s,t}}\right) \text{ for } s\neq t.
\]
The \textbf{Tits representation} $\WW \to \GL(V_\RR)$ is defined on generators $s\in \WW$ by 
\[
    s\cdot v = v-(\alpha_s,v)\alpha_s.
\] 
It is a classical theorem of Tits that this representation is faithful. The set of roots for $\WW$ is 
\[
R:=\{w\cdot \alpha_s \mid w \in \WW, s \in \Gamma_0 \}.
\]
We let $V_\CC':=\Hom_\RR(V_\RR,\CC)$.  The $\WW$-action on $V_\CC'$ is contragradient to that on $V_\RR$: for $Z \in V_\CC', v \in V_\RR$, and $w\in\WW$,
\[
(w\cdot Z)(v) = Z(w^{-1}\cdot v). 
\]
(Here we use $Z$ to denote vectors in $V_\CC'$ since in the sequel regular vectors in $V_\CC'$ will be central charges for stability conditions.)

Let $\iota:\Gamma \to \Gamma$ be the involution induced by $w_0$, i.e. $\alpha_{\iota(s)}=-w_0\cdot\alpha_s$. By Deligne's Theorem \cite[Theorem 4.21]{Del72}, the centre of $\BB$ is infinite cyclic and generated by $\Delta^2$ if $\iota$ is non-trivial, and  by  $\Delta$ otherwise. We define the \textbf{full-twist} $\Theta=\Delta^2$; so $\Theta$ is always central, but whether or not $\Theta$ generates the center of $\BB$ depends on $\Gamma$. 

\subsection{Springer theory for braid groups}

We first recall some notions from Springer theory for Coxeter groups \cite{Springer74}. A vector $Z\in V_\CC'$ is \textbf{regular} if $w\cdot Z=Z$ implies that $w=e$. An element $w\in \WW$ is \textbf{$d$-regular} if it admits a regular eigenvector $Z \in V_\CC'$ with eigenvalue $\zeta_d$, where $\zeta_d$ is a primitive $d$-th root of unity. Finally, an integer $d\in \ZZ_{>0}$ is \textbf{regular} if there exist a $d$-regular element in $\WW$. The classification of regular integers for finite type Coxeter groups is well-known. 

In his seminal work, Springer studied centralisers of regular elements. In particular, he shows that if $w$ is $d$-regular, then its centraliser $\WW_0=C_\WW(w)$ is a complex reflection group acting on the eigenspace $\ker(w-\zeta_d)$.

For a root $\alpha \in R$, define the hyperplane in $V_\CC'$
\[
H_\alpha = \{Z \in V_\CC' \mid Z(\alpha)=0 \}.
\]
Note that $w(H_\alpha) = H_{w\cdot \alpha}$.
The (complexified) \textbf{hyperplane complement} 
\[
\cH  \coloneqq V_\CC'\setminus \bigcup_{\alpha \in R}H_\alpha,
\]
is precisely the locus of regular vectors in $V_\CC'$.
The Coxeter group $\WW$ acts freely on $\cH$, and choosing any $Z \in \cH$ as a basepoint, there is a group isomorphism $\BB \cong \pi_1(\cH/\WW,\ol{Z})$, where $\ol{Z}$ denotes the image of $Z$ in $\cH/\WW$ \cite{VdL_thesis}.
The following loop $\vartheta\in \pi_1(\cH/\WW,\ol{Z})$ plays an important role in what follows:
\begin{align}\label{eq:vartheta}
\vartheta:[0,1] \to \cH/\WW,\; t \mapsto e^{2\pi it}\ol{Z}.
\end{align}
It follows essentially from \cite[Theorem 2.24]{BMR98} that one can choose the isomorphism above so that $\vartheta$ corresponds to the full twist $\Theta$; we give an independent proof of this fact later in Lemma \ref{lem:thetas}.

\begin{Definition}
    A braid $\beta \in \BB$ is \textbf{periodic} if a non-trivial power of it lies in the centre of $\BB$.  Equivalently, $\beta$ is periodic if $\beta^p=\Delta^q$ for some $p,q \in \ZZ, p\neq 0$. 
\end{Definition}

We will now recall Bessis' remarkable ``Springer theory for braid groups'', which is crucial in what follows. One important aspect of this theory is that certain periodic elements -- those that appear as positive integral roots of the full-twist -- play the role of regular elements for $\BB$.  More precisely,

\begin{Theorem}[\cite{Bes15}, Theorem 12.4]\label{thm:bessis}
    Let $d$ be a positive integer, and let $\zeta_d=e^{\frac{2\pi i}{d}}$.
    \begin{enumerate}
        \item There exists a $d$-th root of $\Theta$ if and only if $d$ is regular for $\WW$.
        \item Let $d$ be regular for $\WW$. Then there is a single conjugacy class of $d$-th roots of $\Theta$ in $\BB$.
        \item Let $\beta$ be a $d$-th root of $\Theta$. Let $w$ be the image of $\beta$ in $\WW$. Then, up to conjugation, $w$ is $d$-regular with eigenvalue $\zeta_d$, and the centraliser $C_{\BB}(\beta)$ is isomorphic to the Artin--Tits  group corresponding to the Coxeter group $\WW_0=C_\WW(w)$.
    \end{enumerate}
\end{Theorem}

We refer to \cite[Theorem 7]{Soroko} for a comprehensive list of regular integers in all finite types and a characterisations of periodic elements, although this will not be necessary below.

\subsection{2-Calabi--Yau categories of finite type Coxeter systems}
In this section we recall basic facts about the 2-Calabi--Yau categories associated to finite type Coxeter systems.  
\subsubsection{Simply-laced type}
To begin, we first recall the definition of the 2-Calabi-Yau category when $\Gamma$ is simply-laced and irreducible (that is, of type $A_n,D_n,E_6,E_7,E_8$).
There are several approaches to defining this category, see e.g.\ \cite{BDL2} for an algebraic definition.  We recall here the geometric approach to defining $\cT$, following \cite{Bridgeland09}.  

The McKay correspondence associates to any simply-laced finite type diagram $\Gamma$ a finite subgroup $G$ of $SU(2)$, and a Kleinian singularity $\CC^2/G$; let 
\[
    \pi: Y_\Gamma \longrightarrow \CC^2/G
\]
denote the minimal resolution, and $D^b(Y_\Gamma)$ the bounded derived category of coherent sheaves on $Y_\Gamma$.  The 2-Calabi--Yau category $\cT$ is by definition the full triangulated subcategory of $D^b(Y_\Gamma)$ containing sheaves $E$ such that:
\begin{enumerate}
    \item $E$ is supported on $\pi^{-1}(0)$, and
    \item $R\pi_*E = 0$.
\end{enumerate}
As usual, we let $[1]:\cT\to \cT$ denote the homological shift functor.

The graph $\Gamma$ arises as the
dual graph of the exceptional divisor of $\pi$. The vertices of $\Gamma$ correspond to the irreducible
components of the exceptional divisor (which are copies of $\PP^1$); the edges of $\Gamma$ correspond
to intersection points of the components. For $s\in \Gamma$, let $P_s\in D^b(Y_\Gamma)$ be the sheaf
${i_s}_*\cO(-1)$ where $i_s:\PP^1 \rightarrow Y_\Gamma$ is the inclusion of the component of the exceptional divisor indexed by $s$. The objects $P_s$ form a $\Gamma$-configuration of spherical objects, and $\cT$ is generated by
them as a triangulated category.

The category $\cT$ is 2-Calabi--Yau, meaning that there are isomorphisms $$\hom^\bullet(X,Y) \cong \hom^\bullet(Y,X[2])^*$$ natural in $X,Y$, (here $(-)^*$ denotes the linear dual).  In addition, $\cT$ has a \textbf{standard $t$-structure} whose heart is the extension closure of the $P_s$, and the spherical objects $P_s$ are the simple objects of this heart.  We denote this heart by $\heartsuit_\std$ in what follows.  

The Artin--Tits group $\BB$ acts on $\cT$ (on the left) by autoequivalences, with the generator $\sigma_s$ acting by $\Sigma_{\sigma_s}$, the Seidel-Thomas spherical twist in the spherical object $P_s$ \cite{SeidThom01}. In particular, these autoequivalences satisfy the braid relations \eqref{eq:braidrels} up to isomorphism. Hence, for any $\beta \in \BB$ we have an associated autoequivalence $\Sigma_\beta$, which is well-defined up to isomorphism. 

As a result, $\BB$ acts linearly on the real Grothendieck group $K_0(\cT)\otimes_\ZZ \RR$, with the generator $\sigma_s$ acting by
\[
 [X] \mapsto [X] - \chi  \hom^\bullet (P_s,X) [P_s].
\]
Here $\chi $ denotes the Euler characteristic, so that if $U^\bullet=\oplus_n U_n$ is a $\ZZ$-graded vector space, then 
$\chi U^\bullet = \sum_n (-1)^n \dim U_n$.  In fact 
\[
\chi  \hom^\bullet (P_s,P_s) = 2, \ \ \chi  \hom^\bullet (P_s,P_t)=(\alpha_s,\alpha_t)=-2\cos(\frac{\pi}{m_{s,t}}) \text{ for } s\neq t.
\]
From this, one sees that
\begin{itemize}
\item the action of $\BB$ on $\cT$ induces an action of $\WW$ on $K_0(\cT)\otimes_\ZZ \RR$, and
\item the resulting representation of $\WW$ is isomorphic to the Tits representation $V_\RR$.
\end{itemize}
In particular, there is an isomorphism of representations of $\WW$
\[
V'_\CC \cong \hom_{\ZZ}(K_0(\cT), \CC),
\]
and elements of $V'_\CC$ may be thought of as homomorphisms from the Grothendieck group of $\cT$ to $\CC$.

When $\Gamma$ is simply-laced but not irreducible, so that each connected component of $\Gamma$ is type ADE, we take $\cT$ to be the product of the categories associated to each connected component.

\subsubsection{Non-simply-laced type}
In other finite types the construction of the 2-Calabi--Yau category $\cT$ and the Artin--Tits group action on it is very similar to the one described above; in fact, when $\Gamma$ is not of type ADE, the category $\cT$ associated to $\Gamma$ is an ADE category as above, equipped with the action of a fusion category associated to the non-simply-laced Coxeter system.

More precisely, in \cite{HL24}, we associate to each $\Gamma$ two pieces of data:
\begin{itemize}
    \item a fusion category $\cC$, and
    \item a simply-laced diagram $\widetilde{\Gamma}$, known as the ``unfolding" of $\Gamma$.
\end{itemize}

Since the precise construction is not used directly in any of the proofs below, we do not give the general definition of $\cC$ or the unfolding here, but summarise only the essential points. Interested readers can consult \cite{HL24}  for details.  We stress that if $\Gamma$ is simply-laced, the fusion category $\cC$ is just the category $\vec$ of finite dimensional $\CC$-vector spaces, and the unfolding of $\Gamma$ is $\Gamma$ itself.  

The fusion category $\cC$ and the unfolded diagram $\widetilde{\Gamma}$ are compatible in that there is a right action of $\cC$ on the triangulated category $\cT_{\widetilde{\Gamma}}$ associated to $\widetilde{\Gamma}$.  We define the triangulated category associated to $\Gamma$ to be the \textbf{triangulated module category} $\cT = (\cT_{\widetilde{\Gamma}}, \cC_\Gamma)$; that is, the relevant triangulated category in non-simply-laced type is  a simply-laced type category, {\it together with the action of a fusion category on it}.  (This is not precisely the definition given in \cite{HL24}, but it is equivalent by \cite[Proposition 8.4]{HL24}.)

\begin{example}
Let $\Gamma = I_2(5) = 
\begin{tikzcd}[every arrow/.append style = {shorten <= -.2em, shorten >= -.2em}]
s & t
\arrow["5", no head, from=1-1, to=1-2]
\end{tikzcd}
$.
Then 
\[
\widetilde{\Gamma} = A_4 
	=:
	\begin{tikzcd}[every arrow/.append style = {shorten <= -.2em, shorten >= -.2em}]
	3 & 2 \\
	1 & 4
	\arrow[no head, from=1-1, to=1-2]
	\arrow[no head, from=1-1, to=2-2, crossing over]
	\arrow[no head, from=2-1, to=1-2, crossing over]
	\end{tikzcd}.
\]
The fusion category $\cC$ is given by the Fibonacci fusion category $\Fib$.
This is the unique fusion category that has two (non-isomorphic) simple objects $\one$ and $\Phi$ with fusion rules given by $\Phi \otimes \Phi \cong \one \oplus \Phi$ and $\one$ the monoidal unit \cite{Ostrik_rank2}.
The action of $\Phi \in \cC$ on $\cT$ sends 
\begin{align*}
    P_1 \mapsto P_3, \quad &P_3 \mapsto P_1 \oplus P_3; \\
    P_4 \mapsto P_2, \quad &P_2 \mapsto P_2 \oplus P_4.
\end{align*}
\end{example}

The Grothendieck group $K_0(\cC)$ is a fusion ring, and the Frobenius-Perron map \cite[Definition 3.3.3]{EGNO}
\[
\FPdim: K_0(\cC) \longrightarrow \RR
\]
makes $\RR$ and $\CC$ into $K_0(\cC)$-modules; when $\cC = \vec$, $K_0(\cC) \cong \ZZ$ and $\FPdim$ is just the usual dimension map. 
Thus, in all types, there are isomorphisms of $\WW$-modules
\[
V_\RR \cong K_0(\cT)\otimes_{K_0(\cC)}\RR, \ \ \ V'_\CC \cong \hom_{K_0(\cC)}(K_0(\cT),\CC).
\]
In this way, elements of $V'_\CC$ may be thought of as fusion-ring equivariant homomorphisms from the Grothendieck group of $\cT$ to $\CC$.

For later use, we prove that the full-twist acts by a shift on the 2-CY category $\cT$:

\begin{Proposition}\label{lem:shift2}
    The full-twist autoequivalence $\Theta$ acts on $\cT$ as the Serre functor [2].
\end{Proposition}

\begin{proof}
First suppose that $\Gamma$ is simply-laced.  In \cite{HLLY}, the authors study a triangulated category $\cK^b(A-\mathrm{pmod})$, which categorifies the zero weight space in the adjoint representation of the corresponding quantum group.  If $n$ is the height of the highest root, they show that $\Theta$ acts on $\cK^b(A-\mathrm{pmod})$ by $[2n]\langle 2n-2 \rangle $, where the brackets $\langle \bullet \rangle$ denote internal grading shift and $[\bullet ]$ denotes homological shift (see \cite[Theorems 6.4 \& 7.7 and Lemma 7.4]{HLLY}).
The 2-Calabi-Yau category $\cT$ is not exactly the category $\cK^b(A-\mathrm{pmod})$, but rather is a quotient of it, essentially given by declaring isomorphic the objects $X[1]$ and $ X\langle -1 \rangle$ (see \cite[Section 2.3.3]{BDL2} for the precise construction).  In particular, it follows immediately that in simply-laced type the full twist acts on $\cT$ as shift by two.

Now let $\Gamma$ be any finite-type non-simply-laced (i.e.\ non ADE) Coxeter diagram and $
\widetilde{\Gamma}$ be its unfolding. 
By \cite[Proposition 8.9, Corollary 8.10]{HL24}, the action of $\BB$ on $\cT$ factors through the action of $\BB_{\widetilde{\Gamma}}$ via an injective group homomorphism $\psi: \BB \to \BB_{\widetilde{\Gamma}}$ which sends the positive lift of any Coxeter element of $\BB$ to the positive lift of a Coxeter element of $\BB_{\widetilde{\Gamma}}$ (see \cite{Crisp_LCM} and \cite[\S 8.1]{HL24}).  Since positive lifts of Coxeter elements are $h$-th roots of the full twist, where $h$ is the Coxeter number, it therefore suffices to show that $\WW$ and $\WW_{\widetilde{\Gamma}}$ have the same Coxeter number. 
Indeed, if $\beta_0 \in \BB$ and $\widetilde{\beta}_0 \coloneqq \psi(\beta_0) \in \BB_{\widetilde{\Gamma}}$ are positive lifts of Coxeter elements, then having the same Coxeter number will imply that $\beta_0^h \xmapsto{\psi} \widetilde{\beta}_0^h \cong [2]$, where the isomorphism follows from the simply-laced case above.

The fact that $\WW$ and $\WW_{\widetilde{\Gamma}}$ have the same Coxeter number can be done via a case-by-case analysis, which is listed in \cite[Fig.\ 2]{Heng24} (another proof of this agreement of Coxeter numbers can be found in \cite[Theorem 6.9]{EH_fusionquiver}).
\end{proof}

\subsection{Fusion-equivariant stability conditions on \texorpdfstring{$\cT$}{T}}

In this section we collect some essential properties of the action of $\BB$ on the triangulated category $\cT$ and on its associated stability manifold.  In simply-laced type, the results herein are due to Thomas in type A \cite{Thomas06} and Bridgeland in types ADE \cite{Bridgeland09}. For more general Coxter systems they appear in work of the first two authors \cite{HL24}. 

Recall that a \textbf{(Bridgeland) stability condition} $\tau \coloneqq (\cP,Z)$ on $\cT$ consists of a slicing $\cP$ and a central charge $Z$. 
The slicing $\cP=(\cP(\phi))_{\phi \in \RR}$ is a one-parameter family of additive subcategories $\cP(\phi) \subset \cT$, and the objects in $\cP(\phi)$ are referred to as semistable objects of phase $\phi$. 
The central charge $Z\in \hom_\ZZ(K_0(\cT),\CC)$ is a group homomorphism. 
This datum satisfies several requirements, the most important of which is a compatibility between the slicing and central charge.  We refer the reader to \cite{bridgeland_2007} for the precise definition, and mention only the facts we will need here. 

The slicing $\cP$ in a stability condition is a refinement of the notion of a t-structure on $\cT$. More precisely, the category $\cP(0,1]$, defined as the extension-closure of the semistable objects of phase $\phi$ for $0<\phi\leq1$, is always the heart of a $t$-structure.

Let $\Stab(\cT)$ denote the moduli of stability conditions associated to $\cT$.  By Bridgeland's Theorem \cite{bridgeland_2007}, this is naturally a complex manifold where the local charts are provided by the map taking a stability condition to its central charge. It can be difficult to say much about the stability manifold for arbitrary triangulated categories.  However, in the specific case of interest here, $\Stab(\cT)$ is fairly well understood.  In particular, it is known to be connected \cite{BDL2}.  Moreover, it contains a closed submanifold $\Stab_\cC(\cT)$, which is a distinguished connected component of the space of $\cC$-equivariant stability conditions (\cite{DHL}).  Here a $\cC$-equivariant stability condition $(\cP,Z)$ is a stability condition such that 
\begin{itemize}
    \item each semistable subcategory $\cP(\phi)$ is preseved by the action of the fusion category $\cC$, and
    \item the central charge $Z$ is $K_0(\cC)$ equivariant, i.e.\ $Z \in \Hom_{K_0(\cC)}(K_0(\cT),\CC).$
\end{itemize}
We note again that if $\Gamma$ is ADE, so that $\cC=\vec$, then $\Stab_\cC(\cT)=\Stab(\cT)$ is the entire stability manifold, so, at least in simply-laced type, one can ignore any of the additional structure coming from the action of the fusion category.

An autoequivalence $\Sigma$ of $\cT$ induces a homeomorphism of $\Stab(\cT)$, denoted $g_\Sigma$, and defined as follows: for $\tau=(\cP,Z)$ set $g_\Sigma(\tau)=(\cP',Z')$, where $\cP'(\phi)\coloneqq \Sigma(\cP(\phi))$ and $Z'([X]) \coloneqq Z([\Sigma^{-1}(X)]$.
If $\Sigma$ is an autoequivalence of $\cT$ as a $\cC$-module (i.e. $\Sigma$ commutes with the action of the fusion category), then the submanifold $\Stab_\cC(\cT)$ is preserved by $g_\Sigma$.
When the autoequivalence $\Sigma = \Sigma_\beta$ comes from an element $\beta\in \BB$, we write $g_\beta \coloneqq g_{\Sigma_\beta}$ or simply $\beta\cdot\tau$ for the stability condition $g_\beta(\tau)$. 

The manifold $\Stab_\cC(\cT)$ also has a natural $\CC$-action, which translates the slicing and scales the central charge: for $z = a+\pi i b \in \CC$ and $\tau=(\cP,Z)$ define $z\cdot \tau=(\cP',Z')$, where $\cP'(\phi)=\cP(\phi-b)$ and $Z'=e^zZ$. Note this action commutes with all autoequivalences, i.e. $g_\Sigma(z \cdot \tau) = z\cdot g_\Sigma(\tau)$ for any autoequivalence $\Sigma$. In particular, $\BB$ acts on the quotient space $\Stab_\cC(\cT)/\CC$. 

We summarise the results about stability conditions on $\cT$ that we need in the following:
\begin{Theorem}[\cites{HL24,BravThomas,Bridgeland09}]\label{thm:HLmain}
We have the following:
\begin{enumerate}
    \item $\Stab_\cC(\cT)$ is the universal cover of $\cH/\WW$.
    The covering map $\pi$ is given by sending a $\cC$-equivariant stability condition $\tau=(\cP,Z)$ to $\ol{Z}$.
    \item The group $\BB$ acts faithfully on $\cT$, with the Artin generators acting by fusion-spherical twists.   The induced action of $\BB$ on $\Stab_\cC(\cT)$ is also faithful, and $\beta \mapsto g_\beta$ identifies $\BB$ with the group of deck transformations $\Deck(\pi)$.
    \item The subset of $\cC$-equivariant stability conditions $(\cP,Z)$ such that $\cP(0,1] = \heartsuit_\std$ forms a fundamental domain for the $\BB$-action on $\Stab_\cC(\cT)$.
\end{enumerate}
\end{Theorem}
In particular, the fundamental domain of the $\WW$-action on $\cH$ can be given by the subset of central charges whose image $Z(\alpha_s)$ for all simple roots $\alpha_s$ lies in $\HH \cup \RR_{<0}$, the strict upper half plane union the negative real line.


The deck-transformation action of $\pi_1(\cH/\WW,\ol{Z})$ on its universal cover depends on the choice of lift of the base point $\ol{Z}$.  In what follows we fix this lift to be the unique stability condition $(\cP,Z) \in \pi^{-1}(\ol{Z})$ whose heart $\cP(0,1]$ is $\heartsuit_\std$.
For $\lambda \in \pi_1(\cH/\WW,\ol{Z})$, we let $f_\lambda$ denote the corresponding deck transformation. (Sometimes, we'll  simply write $\lambda\cdot\tau$ for the action of $\lambda$ on a stability condition.)   Thus we have the following isomorphisms:
\begin{alignat}{3}\label{eq:deckiso}
    \BB &\xrightarrow{\cong} \Deck(\pi) &&\xleftarrow{\cong} \pi_1(\cH/\WW,\ol{Z}) \\
    \beta &\mapsto g_\beta, \ \ f_\lambda &&\mapsfrom \lambda,
\end{alignat}
the composition of which gives, for any choice of $Z$, an explicit isomorphism $\BB \cong \pi_1(\cH/\WW,\ol{Z})$.
For the following lemma recall that $\Theta \in \BB$ is the full-twist and $\vartheta \in \pi_1(\cH/\WW,\ol{Z})$ is defined in \eqref{eq:vartheta}.

\begin{Lemma}\label{lem:thetas}
Under the isomorphism \eqref{eq:deckiso} we have that $\Theta \mapsto \vartheta$.
\end{Lemma}

\begin{proof}
Since $\BB$ acts faithfully on $\cT$, it suffices to show that $f_\vartheta = g_\Theta$. 
To prove this equality, note $\vartheta$ lifts to the path $\tilde{\vartheta}$ in $\Stab(\cT)$ given by $t \mapsto (2\pi i t) \cdot \tau_0$.  Since $\tilde{\vartheta}(1)=g_{[2]}(\tau_0)$, we have that $f_\vartheta = g_{[2]}$. On the other hand, by Proposition \ref{lem:shift2} the full-twist autoequivalence $\Theta$ acts on $\cT$ as shift by two.
\end{proof}


In order to check whether or not a braid $\beta\in \BB$ acts by the identity functor on $\cT$, it is sufficient to check that $\Sigma_\beta(P_s) \cong P_s$ for each of the generating spherical objects $P_s$.  (Using the faithfulness of the action of $\BB$ on $\cT$, this fact gives an algorithm for solving the word problem in $\BB$: acting on each of the $P_s$ in turn and check whether or not the result is isomorphic to $P_s$.)  Like the identity functor, other shifts $[a]$ can be detected by checking the action of $\beta$ on generating objects.  We recall this fact for later use here.

\begin{Lemma}[\protect{\cite[Proposition 7.26]{HL24}}]\label{lem:rigid2}
Let $\beta \in \BB$. If there exists $a\in \ZZ$ such that for every $s \in \Gamma_0$, $\Sigma_\beta(P_s) \cong P_s[a]$, then $\Sigma_\beta\cong [a]$.   
\end{Lemma}

\subsection{The \texorpdfstring{$\CC$}{C}-action on \texorpdfstring{$\cT$}{T} and t-exactness}

Theorem \ref{thm:main2} concerns elements of $\BB$ that have a a fixed point in this space, i.e. $\beta \in \BB$ such that $\beta \cdot \tau = z \cdot \tau$ for some $\tau \in \Stab_\cC(\cT)$ and $z\in\CC$ . We explain how to interpret these fixed points from a homological point of view. Let us first recall a classical notion:

\begin{Definition}\label{defn:t-exact}
    Let $\Psi: \cT \to \cT$ be an exact endofunctor and fix a heart $\heartsuit$ on $\cT$.
    The functor $\Psi$ is said to be \textit{t-exact} (with respect to $\heartsuit$) if $\Psi(\heartsuit) \subseteq \heartsuit$; it is \textit{t-exact up to a shift $m \in \ZZ$} if $\Psi(\heartsuit) \subseteq \heartsuit[m]$.
\end{Definition}

Now suppose that $\Psi\cdot \tau =z\cdot \tau$ for some $\tau \in \Stab_\cC(\cT)$ and $z=-m\pi i, m \in \ZZ$. Then $\Psi (\cP(0,1]) = \cP(m,m+1]$, and hence $\Psi$ is t-exact up to shift $m$ with respect to the standard t-structure induced by $\cP$. From this perspective, a functor fixing a point in $\Stab_\cC(\cT)/\CC$ is a refinement of the notion of t-exactness.



\begin{example}
Let $\Gamma = A_3$. We consider some examples of autoequivalences scaling a stability condition arising from the action of $\BB$ on $\Stab(\cT)=\Stab_\cC(\cT)$. 
\begin{enumerate}
    \item The central element $\Theta$ acts by $[2]$ on $\cT$, and hence acts by $2\pi i \in \CC$ on $\Stab_\cC(\cT)$.  Thus all points of $\Stab_\cC(\cT)/\CC$ are fixed by $\Theta$.
    \item Consider first the periodic element $\beta \coloneqq \sigma_1\sigma_2\sigma_3\sigma_1$  in $\BB$.
    Define a stability condition $\tau \in \Stab(\cT)$ as in Figure \ref{fig:centralchargefixpt} (left).  In this stability condition the generating objects $P_i$ are themselves stable, and we have drawn the central charge in the picture, putting the object $P_i$ on the corresponding central charge vector.  In this way the labelling in the figure defines both the slicing and the central charge. Note that the associated heart $\cP(0,1]$ of $\tau$ is $\heartsuit_\std$, since $Z([P_i]) \in \HH \cup \RR_{<0}$ for $i=1,2,3$. (The unlabelled vectors are the central charges of other stable objects.)
    
    A straightforward computation \cite[Section 2.7]{LQ} shows that $\beta$ acts on $\tau$ by the complex number $\frac{2\pi i}{3}$ as depicted in Figure \ref{fig:centralchargefixpt} (right).
    In other words, $\beta \cdot \tau = \frac{2\pi i}{3}\cdot \tau$ and $\tau \in \Stab(\cT)/\CC$ is a fixed point of $\beta$.
    \item Now consider the periodic element $\beta'\coloneqq \sigma_1\sigma_2\sigma_3$ in $\BB$, and the stability condition given by Figure \ref{fig:centralchargefixpt2} (left).  In this case, the heart $\cP(0,1]$ of $\tau$ is given by $\sigma_3^{-1}(\heartsuit_\std)$ and the action of $\beta'$ rotates $\tau$ by $\frac{\pi}{2}$ as in Figure \ref{fig:centralchargefixpt2} (right), i.e. $\beta'\cdot\tau' = \frac{\pi i}{2}\cdot\tau'$.
\end{enumerate}
In both examples the element in $\BB$ rotates the stability condition by a scalar of the form $x \pi i$, for some $x \in \QQ$. This is a general phenomenon (Lemma \ref{lem:rotaterational}).
Note also that $\beta^3 = {\beta'}^4 = \Theta$ is the full-twist, which acts on $\cT$ by the triangulated shift $[2]$ (Proposition \ref{lem:shift2}). In other words, a power of these elements is t-exact up to shift. This is also a general phenomenon: those autoequivalences which scale stability conditions in $\Stab_\cC(\cT)$ are \textit{roots of t-exact autoequivalences} (cf. Section \ref{sec:backwards}).
\begin{figure}
    \centering
    \begin{tikzpicture}[scale = 1.5]
    \coordinate (Origin)   at (0,0);
    \coordinate (XAxisMin) at (-1.4,0);
    \coordinate (XAxisMax) at (1.4,0);
    \coordinate (YAxisMin) at (0,-1.6);
    \coordinate (YAxisMax) at (0,1.6);

    \draw [thin, gray!30,-latex] (XAxisMin) -- (XAxisMax);
    \draw [thin, gray!30,-latex] (YAxisMin) -- (YAxisMax);

    \foreach \i in {0, 60, ..., 300} {
	   \draw[blue, thick, ->] (Origin) -- ++(\i:1);  
    }

    \foreach \i in {30, 90, ..., 330} {
	   \draw[blue, thick, ->] (Origin) -- ++(\i:1.6);
    }

    \node (a1) at ($(Origin) + (120:1.2)$)
        {$P_1$};
    \node (a2) at ($(Origin) + (180:1.2)$)
        {$P_2$};
    \node (a3) at ($(Origin) + (30:1.8)$)
        {$P_3$};

    \node (map) at (2.5,0) {$\xmapsto{\sigma_1\sigma_2\sigma_3\sigma_1}$};
    
    \end{tikzpicture}
    \begin{tikzpicture}[scale = 1.5]
    \coordinate (Origin)   at (0,0);
    \coordinate (XAxisMin) at (-1.4,0);
    \coordinate (XAxisMax) at (1.4,0);
    \coordinate (YAxisMin) at (0,-1.6);
    \coordinate (YAxisMax) at (0,1.6);

    \draw [thin, gray!30,-latex] (XAxisMin) -- (XAxisMax);
    \draw [thin, gray!30,-latex] (YAxisMin) -- (YAxisMax);
    	
	\begin{scope}[rotate=120]
    \foreach \i in {0, 60, ..., 300} {
	   \draw[blue, thick, ->] (Origin) -- ++(\i:1);  
    }

    \foreach \i in {30, 90, ..., 330} {
	   \draw[blue, thick, ->] (Origin) -- ++(\i:1.6);
    }

    \node (a1) at ($(Origin) + (120:1.2)$)
        {$P_1$};
    \node (a2) at ($(Origin) + (180:1.2)$)
        {$P_2$};
    \node (a3) at ($(Origin) + (30:1.9)$)
        {$P_3$};
	    
	\end{scope}
	
    \end{tikzpicture}
    \caption{Central charge of the stable objects $P_s$ and their action under $\beta =\sigma_1\sigma_2\sigma_3\sigma_1$.}
    \label{fig:centralchargefixpt}
\end{figure}
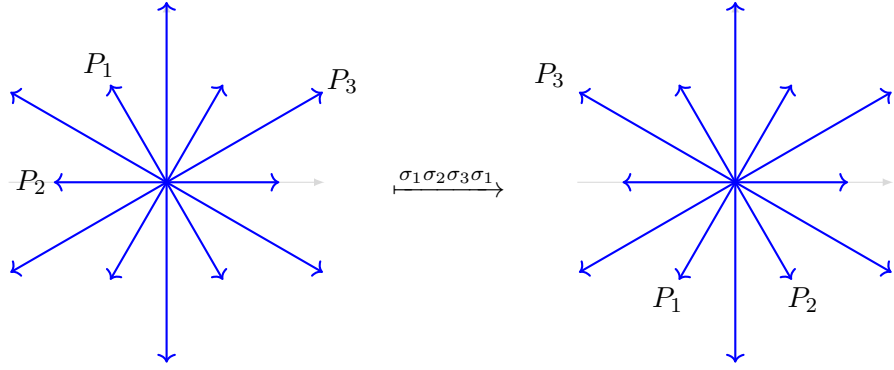

\begin{figure}
    \centering
  \begin{tikzpicture}[scale = 1.5]
    \coordinate (Origin)   at (0,0);
    \coordinate (XAxisMin) at (-1.4,0);
    \coordinate (XAxisMax) at (1.4,0);
    \coordinate (YAxisMin) at (0,-1.4);
    \coordinate (YAxisMax) at (0,1.4);

    \draw [thin, gray!30,-latex] (XAxisMin) -- (XAxisMax);
    \draw [thin, gray!30,-latex] (YAxisMin) -- (YAxisMax);

    \foreach \i in {0, 90, ..., 270} {
	   \draw[blue, thick, ->] (Origin) -- ++(\i:1);  
    }

    \foreach \i in {45, 135, ..., 315} {
	   \draw[blue, thick, ->] (Origin) -- ++(\i:1.414);
    }

    \node (a1) at ($(Origin) + (180:1.2)$)
        {$P_1$};
    \node (a2) at ($(Origin) + (90:1.2)$)
        {$P_2$};
    \node (a3) at ($(Origin) + (0:1.2)$)
        {$P_3$};
    \node (map) at (2.2,0) {$\xmapsto{\sigma_1\sigma_2\sigma_3}$};
  \end{tikzpicture}
  \begin{tikzpicture}[scale = 1.5]
    \coordinate (Origin)   at (0,0);
    \coordinate (XAxisMin) at (-1.4,0);
    \coordinate (XAxisMax) at (1.4,0);
    \coordinate (YAxisMin) at (0,-1.4);
    \coordinate (YAxisMax) at (0,1.4);

    \draw [thin, gray!30,-latex] (XAxisMin) -- (XAxisMax);
    \draw [thin, gray!30,-latex] (YAxisMin) -- (YAxisMax);
    	
	\begin{scope}[rotate=90]
    \foreach \i in {0, 90, ..., 270} {
	   \draw[blue, thick, ->] (Origin) -- ++(\i:1);  
    }

    \foreach \i in {45, 135, ..., 315} {
	   \draw[blue, thick, ->] (Origin) -- ++(\i:1.414);
    }

    \node (a1) at ($(Origin) + (180:1.2)$)
        {$P_1$};
    \node (a2) at ($(Origin) + (90:1.2)$)
        {$P_2$};
    \node (a3) at ($(Origin) + (0:1.2)$)
        {$P_3$};

    \end{scope}
	
  \end{tikzpicture}
    \caption{Central charge of the stable objects $P_s$ and their action under  $\beta'=\sigma_1\sigma_2\sigma_3$.}
    \label{fig:centralchargefixpt2}
\end{figure}
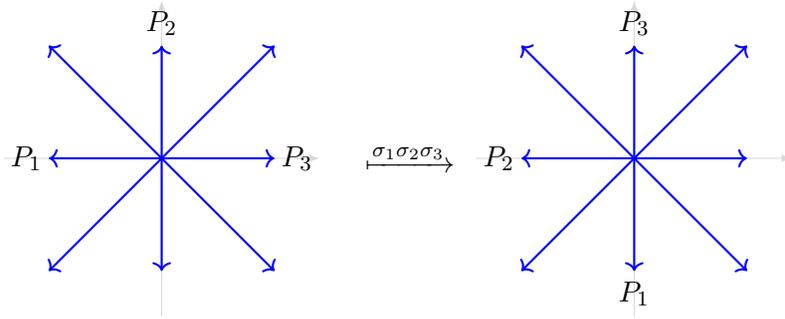
\end{example}



\section{Proof of the main theorem}

This section is devoted to proving Theorem \ref{thm:main2}.

\subsection{Forward implication using Springer theory}

First we prove the forward direction of the theorem:

\begin{Proposition}\label{prop:direct1}
    If $\beta \in \BB$ is periodic then there is a point in $\Stab_\cC(\cT)/\CC$ that it is fixed by $\beta$.
\end{Proposition}

Bessis' Springer theory for braid groups allows us to make an immediate simplification:

\begin{Lemma}\label{lem:dregular}
    If for every regular integer $d$ there exists  a $d$-th root of $\Theta$ that fixes a point in $\Stab_\cC(\cT)/\CC$, then 
    Proposition \ref{prop:direct1} follows.
\end{Lemma}

\begin{proof}
    Note that if $\beta$ has a fixed point in $\Stab_\cC(\cT)/\CC$ then so does any power of $\beta$. Therefore it suffices to prove Proposition \ref{prop:direct1} for 
    primitive periodic elements (recall $\beta$ is 
    primitive if $\beta=\gamma^k$ implies $k=\pm 1$ for all $\gamma \in \BB$). 
By \cite[Theorem 3.14]{LL11} any primitive periodic element is a root of $\Theta$. 

    
    So now we've reduced the problem to finding a fixed point for $d$-th roots of $\Theta$, where $d$ is any positive integer. 
    By part (i) of Theorem \ref{thm:bessis}, $d$ is regular for $\WW$. Now if $\tau \in \Stab_\cC(\cT)/\CC$ is a fixed point for $\beta$, then $\gamma\cdot\tau$ is a fixed 
    point for $\gamma\beta\gamma^{-1}$. Moreover, by part (ii) of Theorem \ref{thm:bessis}, there is a single conjugacy class of $d$-th roots of $\Theta$ in $\BB$. 
    Hence, it suffices to find, for every regular integer $d$, a fixed point for some $d$-th root of $\Theta$.
\end{proof}

Now fix a regular integer $d$. By Lemma \ref{lem:dregular}, to prove Proposition \ref{prop:direct1} it suffices to find a $d$-th root of $\Theta$ which fixes some stability condition in $\Stab_\cC(\cT)/\CC$.  Let $\zeta_d=e^{\frac{2\pi i}{d}}$. By classical Springer theory, there exists $w \in \WW, Z \in \cH$ such that $w\cdot Z = \zeta_dZ$. The image $\ol{Z} \in \cH/\WW$ of the eigenvector $Z$ provides the central charge for our sought-after stability condition. To lift it to a stability condition, we use that $\pi:\Stab_\cC(\cT) \to \cH/\WW$ is a (universal) cover (see Theorem \ref{thm:HLmain}(1)). In the fiber $\pi^{-1}(\ol{Z})$, we choose the unique stability condition $\tau_0$ whose heart $\cP(0,1]$ is the standard heart $\heartsuit_\std$ (Theorem \ref{thm:HLmain}(3)).  We'll now show that $\tau_0$ is fixed point for a $d$-th root of $\Theta$.  

Bessis introduced the \textit{standard $d$-th root} $\lambda \in \pi_1(\cH/\WW,\ol{Z})$ given by $t \mapsto e^{2\pi i\frac{t}{d}}\ol{Z}$. Note that $\lambda(1)=\zeta_d \ol{Z}= \ol{w\cdot Z}$, so this is indeed a loop in $\cH/\WW$. By Theorem \ref{thm:HLmain}(2), there exists $\beta \in \BB$ such that $f_{\lambda}=g_\beta$. 
By Lemma \ref{lem:thetas} we have that  
$
    f_\vartheta = g_\Theta.
$
Hence, we have that 
    \begin{align*}
        g_{\beta^d}  = f_{\lambda^d} = f_\vartheta = g_\Theta,
    \end{align*}
and since $\BB$ acts faithfully on $\Stab(\cT)$ (Theorem \ref{thm:HLmain}(2)), it follows that $\beta^d = \Theta$.

So $\beta$ is a $d$-th root of $\Theta$, and we'll now show it fixes $\tau_0 \in \Stab_\cC(\cT)/\CC$. 
First, lift $\lambda$ to a path $\tilde{\lambda}$ in $\Stab_\cC(\cT)$ beginning at $\tau_0$, and let $\tau_1$ be its endpoint. 
By construction, $\tau_1 = \lambda\cdot \tau_0 = \beta\cdot \tau_0$. On the other hand, the natural action of $\CC$ on $\Stab_\cC(\cT)$ lifts the action of $\CC^\times$ on $\cH$: $\pi(z\cdot\tau)=e^z\pi(\tau)$. Since all the points on the loop $\lambda$ are in a single $S^1$ orbit, all the stability conditions in the path $\tilde{\lambda}$ from $\tau_0$ to $\tau_1$ are in the same $\mathbb{R}$ orbit. Hence, $\tau_1=z\cdot \tau_0$ for some $z \in \RR$ (in fact, $z=\frac{1}{d}$ but we don't need that here). We've now shown that $\beta\cdot \tau_0 = z\cdot \tau_0$. This completes the proof of Proposition \ref{prop:direct1}.

\begin{Remark}\label{rem:direct1}
Observe that our proof of Proposition \ref{prop:direct1} did not use any intrinsic properties of the space $\Stab_\cC(\cT)$, and indeed this argument would work for any universal cover of $\cH/\WW$. More precisely, if $\pi:Y \to \cH/\WW$ is any universal cover, then $Y$ is naturally endowed with a $\CC$-action since $\CC^\times$ acts on $\cH/\WW$. Now choose any point $\ol{Z} \in \cH/\WW$ and any isomorphism $\BB \to \pi_1(\cH/\WW,\ol{Z})$, thereby endowing $Y$ with a $\BB$-action (which commutes with the $\CC$-action). Then the same proof as above gives: if $\beta$ is periodic then $\beta$ has a fixed point in $Y/\CC$. 
On the other hand, our argument for the other direction of Theorem \ref{thm:main2} uses the modular description of the universal cover as a space of stability conditions (cf.\ Remark \ref{rem:direct2}).
\end{Remark}

\subsection{Backward implication}\label{sec:backwards}

To complete the proof of Theorem \ref{thm:main2} we prove:

\begin{Proposition}\label{prop:direct2}
    If $\beta \in \BB$ has a fixed point in $\tau \in \Stab_\cC(\cT)/\CC$, i.e. $\beta\cdot \tau = z\cdot \tau$ for some $z \in \CC$, then $\beta$ is periodic.
\end{Proposition}

By conjugating $\beta$ if necessary, we can assume without loss of generality that $\tau$ is a standard stability condition. 
The first step is to restrict the set of $z$ that can arise in Proposition \ref{prop:direct2}:

\begin{Lemma}\label{lem:rotaterational}
In the setting of Proposition \ref{prop:direct2}, we have  $z\in \QQ\pi i$.
\end{Lemma}

\begin{proof}
Recall that $R$ denotes the set of roots, and let $w \in \WW$ be the image of $\beta$. Consider the following collection of finite points on the circle $$X_\tau=\{e^{\pi i \phi_\alpha} \mid \alpha \in R\} \subset S^1,$$ where  $e^{\pi i \phi_\alpha}$ is defined by $Z_\tau(\alpha)=m_\alpha e^{\pi i \phi_\alpha}$. Now, on the one hand, $X_{\tau} = X_{\beta\cdot\tau}$ since $Z_{\beta\cdot\tau}(\alpha)=Z_\tau(w^{-1}(\alpha))$ and the Coxeter group action permutes $R$. On the other hand, $X_{z\cdot\tau}=e^{z}X_\tau$. Hence, since $\beta\cdot \tau = z\cdot \tau$, we have that 
$X_{\tau}=e^{z}X_\tau$. Therefore multiplication by $e^{z}$ defines a permutation of $X_\tau$. Since this set is finite, this permutation has finite order, so that $(e^{z})^k=1$ for some $k$, i.e. $z \in \QQ\pi i$.
\end{proof}

Using the lemma above we write $z=\frac{p\pi i}{q}$. Then 
\[
\Sigma_\beta^q[p](\cP(\varphi)) = \cP(\varphi-p)[p] =\cP(\varphi),
\]
where $\cP$ is the slicing of $\tau$. 
The equation above implies that $\Sigma_\beta^q \cong \Sigma_{\beta^q}$ is t-exact up to shift by $-p$ (see Definition \ref{defn:t-exact}) with respect to the heart $\cP(0,1] = \heartsuit_\std$. The following lemma then shows that $\beta^q$ is periodic. This implies that $\beta$ is periodic as well, completing the proof of Proposition \ref{prop:direct2}.

\begin{Lemma}\label{lem:t-exact-and-periodic}
Let $\beta \in \BB$ and suppose that $\Sigma_{\beta}$ is t-exact up to shift (with respect to the standard t-structure).  Then $\beta$ is periodic.  
\end{Lemma}
\begin{proof}
Recall that  $\heartsuit_\std \subset \cT$ denotes the heart of the standard t-structure, which contains simple objects $P_s$ for $s\in \Gamma_0$. Since $\Sigma_{\beta}$ is $t$-exact up to some shift $a \in \ZZ$, $\Sigma_{\beta}[-a]:\heartsuit_\std \to \heartsuit_\std$ is an equivalence of abelian categories. Hence it induces a bijection $f:\Irr(\heartsuit_\std) \to \Irr(\heartsuit_\std)$ on the set of isomorphism classes of simple objects.

Since $\Irr(\heartsuit_\std)$ is finite, there exists $k \in \ZZ_{>0}$ such that $f^k=id$.  Hence there exists $d \in \ZZ$ (namely $d=ak$) such that $\Sigma_{\beta^k}(P_s)\cong P_s[d]$ for all simple objects $P_s\in \heartsuit_\std$.  By Lemma \ref{lem:rigid2} this implies that $\Sigma_{\beta^k} \cong [d]$.  Hence for any $\gamma \in \BB$, $\Sigma_{\gamma}\Sigma_{ \beta^k}  \cong \Sigma_{\beta^k}\Sigma_{\gamma}$.  By Theorem \ref{thm:HLmain}(2) this implies that $\beta^k \in Z(\BB)$, i.e. $\beta$ is periodic.
\end{proof}

\begin{Remark}\label{rem:direct2}
    Lemma \ref{lem:t-exact-and-periodic} is the only part of our argument where the stability space as a moduli space of categorical structures is used in an essential way. In particular, both the categorical rigidity Lemma \ref{lem:rigid2} and faithfulness of the categorical action (Theorem \ref{thm:HLmain}(2)) are used in the above argument.
\end{Remark}

\section{Future work}\label{sec:final}

Motivation for this work comes from an ongoing program to develop a dynamical theory for autoequivalence groups of triangulated categories.  A model theorem from topology is the Nielsen--Thurston classification \cite{Thur88}, which gives a dynamical classification of elements in the mapping class group of a (compact orientable) surface. In this classification, two dynamical classes -- periodic and pseudo-Anosov -- form the basic building blocks used to understand the dynamics of an arbitrary mapping class. Moreover, the action of the mapping class group on the Teichm\"uller space of the surface plays an important role in this theory: periodic maps have a fixed point in Teichm\"uller space whereas pseudo-Anosov maps have a pair of fixed points in its Thurston boundary.

The Nielson--Thurston classification does not directly give rise to a dynamical classification of elements in finite type Artin--Tits groups.  Indeed, finite type Artin--Tits groups typically do not arise as (subgroups of) surface mapping class groups \cite{Waj_typeE}. However, Theorem \ref{thm:main2} is a step towards an analogous categorical Nielsen--Thurston classification, which should apply not just to finite type Artin--Tits groups but also to a larger family of triangulated autoequivalence groups, such as the autoequivalence group of the derived category of a K3 surface; see e.g.\ \cite[Theorem 1.4]{FL_NielsenRealization}, which studies fixed points in $\Stab(\cD)/\CC$ in relation to finite subgroups of autoequivalences.

In this sense, Theorem \ref{thm:main2} is part of the larger program to understand autoequivalence groups of triangulated categories with metrics.  In this program, the triangulated category itself is thought of as something akin to the topological Fukaya category of a surface, and, in the presence of a triangulated metric such as a stability condition, an autoequivalence is studied via the discrete dynamical system it defines \cite{DHKK}.  Next steps in this program should involve developing the dynamical theory of elements in $\BB$ which give rise to pseudo-Anosov autoequivalences \cite{FFHKL} and studying the action of these autoequivalences on the compactifications of \cite{BDL}.


\bibliographystyle{alpha}
\bibliography{monbib}

\end{document}